\documentclass[twoside,12pt]{article}
\usepackage[all, 2cell, dvips]{xy}
\usepackage{latexsym,amsmath, amsfonts, graphics, epsf, epic}
\usepackage{amsthm}
\usepackage{amssymb}
\usepackage{amsopn}
\usepackage{amscd}
\usepackage{color}

\theoremstyle{plain}
\newtheorem{thm}{Theorem}[section]

\newtheorem{lem}[thm]{Lemma}
\newtheorem{prop}[thm]{Proposition}

\theoremstyle{definition}

\newtheorem{remark}[thm]{Remark}
\newtheorem{example}[thm]{Example}
\newtheorem{defn}[thm]{Definition}

\usepackage{amssymb}

\def\lm{\lambda}

\def\al{\alpha}

\def\l.l.o.{\it l.l.o}

\def\chiup{\raise 2pt\hbox{$\chi$}}

\title{Identities for the number of standard Young
tableaux in some $(k,\ell)$ hooks}

\begin{document}
\maketitle
\centerline{A. Regev}

\medskip {\bf Abstract}: Closed formulas are known for $S(k,0;n)$,
the number of standard Young tableaux of size $n$ and with at most
$k$ parts, where $1\le k\le 5$. Here we study the analogue problem
for $S(k,\ell;n)$, the number of standard Young tableaux of size
$n$ which are contained in the $(k,\ell)$ hook. We deduce some
formulas for the cases $k+\ell\le 4$.

\medskip
2010 Mathematics Subject Classification 05C30

\section{Introduction}
Given a partition $\lm$ of $n$,  $\lm\vdash n$, let $\chi^\lm$
denote the corresponding irreducible $S_n$ character. Its degree
is denoted by $\deg \chi^\lm=f^\lm$ and is equal to the number of
Standard Young tableaux (SYT) of shape
$\lm$~\cite{kerber},~\cite{macdonald},~\cite{sagan},~\cite{stanley}.
The number $f^\lm$ can be calculated for example by the hook
formula~\cite[Theorem 2.3.21]{kerber},~\cite[Section
3.10]{sagan},~\cite[Corollary 7.21.6]{stanley}. We consider the
number of  SYT in the $(k,\ell)$ hook. More precisely, given
integers $k,\ell,n\ge 0$ we denote
\[
H(k,\ell;n)=\{\lm=(\lm_1,\lm_2,\ldots)\mid \lm\vdash n~\mbox{and}~
\lm_{k+1}\le \ell\}\qquad\mbox{and}\qquad S(k,\ell;n)=\sum_{\lm\in
H(k,\ell;n)}f^\lm.
\]
\subsection{The cases where $S(k,\ell;n)$ are known}\label{s1}
For the "strip" sums $S(k,0;n)$ it is
known~\cite{regev1}~\cite{stanley} that
\[
S(2,0;n)={n\choose\lfloor\frac{n}{2}\rfloor}\quad\mbox{and}\quad
S(3,0;n)=\sum_{j\ge 0}\frac{1}{j+1}{n\choose 2j}{2j\choose j}.
\]
Let $C_j=\frac{1}{j+1}{2j\choose j}$ be the Catalan numbers, then
Gouyon-Beauchamps~\cite{gouyon}~\cite{stanley} proved that
\[
S(4,0;n)=C_{\lfloor\frac{n+1}{2}\rfloor}\cdot
C_{\lceil\frac{n+1}{2}\rceil}\quad\mbox{and}\quad
S(5,0;n)=6\sum_{j=0}^{\lfloor\frac{n}{2}\rfloor}{n\choose 2j}\cdot
C_j\cdot\frac{(2j+2)!}{(j+2)!(j+3)!}.
\]
\medskip
As for the "hook" sums, until recently only $S(1,1;n)$ and
$S(2,1;n)=S(1,2;n)$ have been calculated:

\medskip
1. ~It easily follows that $S(1,1;n)=2^{n-1}$.

\medskip
2. ~The following identity was proved in~\cite[Theorem
8.1]{regev2}:
\begin{eqnarray}\label{motzkin.path.3}
S(2,1;n)=~~~~~~~~~~~~~~~~~~~~~~~~~~~~~~~~~~~~~~~~~~~~~~~~~~~~~~~~~~~~~~~~~~~~~~~~~~~~~~~~~~~~~~~~~~~~
\end{eqnarray}
\[
~~~~~~=\frac{1}{4}\left(\sum_{r=0}^{n-1}{n-r\choose{\lfloor\frac{n-r}{2}\rfloor}}
{n\choose r}
+\sum_{k=1}^{\lfloor\frac{n}{2}\rfloor-1}\frac{n!}{k!\cdot
(k+1)!\cdot (n-2k-2)!\cdot (n-k-1)\cdot(n-k)}\right)+1.
\]

\medskip
\subsection{The main results}
In Section~\ref{s2} we prove Equation~\eqref{rewrite8}, which
gives (sort of) a closed formula for $S(3,1;n)$ in terms of the
Motzkin-sums function. For the Motzkin-sums function
see~\cite[sequence A005043]{sloane}.  Equation~\eqref{rewrite8} in
fact is a "degree" consequence of a formula of $S_n$ characters,
of interest on its own, see Equation~\eqref{rewrite3}.

\medskip
In Section~\ref{s3} we find some intriguing relations between the
sums $S(4,0;n)$ and the "rectangular" sub-sums $S^*(2,2,;n)$, see
below identities~\eqref{b3} and~\eqref{b4}.

\medskip

Finally, in Section~\ref{s4} we review some cases where the
hook-sums $S(k,\ell;n)$ are related, in some rather mysterious
ways, to humps calculations on Dyck and on Motzkin paths,
see~\eqref{eq1},~\eqref{eq2}, and Theorem~\ref{motzkin.humps.1}.

\medskip
As usual, in some of the above identities it is of interest to
find bijective proofs, which might explain these identities.

\medskip
{\bf Acknowledgement.} We thank D. Zeilberger for verifying some
of the identities here by the WZ method.

\section{The sums $S(3,1;n)$ and the characters $\chi
(3,1;n)$}\label{s2}

 Define the $S_n$ character
\begin{eqnarray}\label{rewrite5}
\chi(k,\ell;n)=\sum_{\lm\in H(k,\ell;n)}
\chi^\lm\qquad\mbox{so}\qquad\deg(\chi(k,\ell;n))=S(k,\ell;n).
\end{eqnarray}
\subsection{The Motzkin-sums function}

Define the $S_n$ character
\begin{eqnarray}\label{rewrite6}
\Psi(n)=\sum_{k=0}^{\lfloor
n/2\rfloor}\chi^{(k,k,1^{n-2k})}\qquad\mbox{and
denote}\qquad\deg\Psi(n)=a(n).
\end{eqnarray}
We call $\Psi(n)$ {\it the Motzkin-sums} character. Note that
\[
\deg\chi^{(k,k,1^{n-2k})}=f^{(k,k,1^{n-2k})}=\frac{n!}{(k-1)!\cdot
k!\cdot (n-2k)!\cdot (n-k)\cdot (n-k+1)},
\] hence
\begin{eqnarray}\label{a2}
a(n)=\sum_{k=1}^{\lfloor{n}/{2}\rfloor}\frac{n!}{(k-1)!\cdot
k!\cdot (n-2k)!\cdot (n-k)\cdot (n-k+1)}.
\end{eqnarray}
By~\cite[sequence A005043]{sloane} it follows that $a(n)$ is the
Motzkin-sums function. The reader is referred to~\cite{sloane} for
various properties of $a(n)$. For example, $a(n)+a(n+1)=M_n$,
where $M_n$ are the Motzkin numbers. Also $a(1)=0,~a(2)=1$ and
$a(n)$ satisfies the recurrence:
\begin{eqnarray} \mbox{for $n\ge 3$}\qquad \label{a1} a(n)=\frac{n-1}{n+1}\cdot(2\cdot
a(n-1)+3\cdot a(n-2)).
\end{eqnarray}

 Note also that for $n\ge 2~$ Equation~\eqref{motzkin.path.3} can be
written as
\begin{eqnarray}\label{a3}
S(2,1;n)
=\frac{1}{4}\left(\sum_{r=0}^{n-1}{n-r\choose{\lfloor\frac{n-r}{2}\rfloor}}
{n\choose r}+a(n)-1\right)+1.
\end{eqnarray}

The asymptotic behavior of $a(n)$ can be deduced from that of
$M_n$. We deduce it here, even though it is not needed in the
sequel.

\begin{remark} As $n$ goes to infinity,
\[
a(n)\simeq \frac{\sqrt 3}{8\cdot\sqrt{2\pi}}\cdot\frac{1}{n\sqrt
n}\cdot 3^n\qquad\mbox{and}\qquad a(n)\simeq\frac{1}{4} \cdot M_n.
\]
\begin{proof}
By standard techniques it can be shown that $a(n)$ has asymptotic
behavior $$a(n)\simeq c\cdot \left(\frac{1}{n}\right)^g\cdot
\al^n$$ for some constants $c,g$ and $\al$ -- which we now
determine. By~\cite{regev1}
\[
M_n\simeq
\frac{\sqrt 3}{2\sqrt{2\pi}}
\cdot\left(\frac{1}{n}\right)^{3/2}\cdot 3^n.
\]
With
\[
M_n=a(n)+a(n+1)\simeq c\cdot
(1+\al)\cdot\left(\frac{1}{n}\right)^g\cdot\al^n
\]
this implies that $\al=3$, that $g=3/2$ and that $c=\frac{\sqrt
3}{8\cdot\sqrt{2\pi}}$.

\end{proof}
\end{remark}

\subsection{The outer product of $S_m$ and  $S_n$ characters}

Given an $S_m$ character $\chi_m$ and  an $S_n$ character
$\chi_n$, we can form their {\it outer} product $\chi_n\hat\otimes
\chi_n$. The exact decomposition of $\chi_m\hat\otimes \chi_n$ is
given by the Littlewood-Richardson
rule~\cite{kerber},~\cite{macdonald},~\cite{sagan},~\cite{stanley}.
In the special case that $\chi_n=\chi^{(n)}$, this decomposition
is given, below, by Young's rule. Also
\begin{eqnarray}\label{rewrite7}
\deg (\chi_n\hat\otimes \chi^{(n)})=\deg (\chi_n)\cdot{n+m\choose
n}.
\end{eqnarray}

\medskip {\bf Young's Rule}~\cite{macdonald}: Let $\lm=(\lm_1,\lm_2,\ldots)\vdash m$
and denote by $\lm^{+n}$ the following set of partitions of $m+n$:
\[
\lm^{+n}=\{\mu\vdash n+m\mid \mu_1\ge \lm_1\ge \mu_2\ge
\lm_2\ge\cdots\}.
\]
Then
\[
\chi^\lm\hat\otimes \chi^{(n)}=\sum_{\mu\in\lm^{+n}}\chi^\mu.
\]
\begin{example}\label{rewrite4}~\cite{regev1},~\cite{stanley}
Given $n$, it follows that
\begin{eqnarray}\label{rewrite1}
\chi^{(\lfloor n/2\rfloor)}\hat\otimes \chi^{(\lceil
n/2\rceil)}=\chi(2,0;n),\quad\mbox{and by taking degrees,}\quad
S(2,0;n)={n\choose \lfloor n/2 \rfloor}.
\end{eqnarray}
\end{example}

\subsection{A character formula for $\chi(3,1;n)$}

\begin{prop}\label{rewrite2}
With the notations of~\eqref{rewrite5} and~\eqref{rewrite6},
\begin{eqnarray}\label{rewrite3}
\chi(3,1;n)=\frac{1}{2}\cdot\left[\chi(2,0,n)+\sum_{j=0}^n
\Psi(j)\hat\otimes \chi^{(n-j)} \right].
\end{eqnarray}
By taking degrees, Example~\ref{rewrite4} together
with~\eqref{rewrite6} and~\eqref{rewrite7} imply that
\begin{eqnarray}\label{rewrite8}
S(3,1;n)=\frac{1}{2}\cdot
\left[{n\choose\lfloor\frac{n}{2}\rfloor}+\sum_{j=0}^n
a(j)\cdot{n\choose j} \right].
\end{eqnarray}
\end{prop}
\begin{proof}
Denote
\[
\Omega(n)=\sum_{j=0}^n \Psi(j)\hat\otimes \chi^{(n-j)}
\]
and analyze this $S_n$ character. Young's rule implies the
following:

\medskip

Let $\mu\vdash n$, then by Young's rule $\chi^\mu$ has a positive
coefficient  in $\Omega(n)$ if and only if $\mu\in H(3,1;n)$.
Moreover, all these coefficients are either $1$ or $2$, and such a
coefficient equals $1$ if and only if $\mu$ is a $\le 2$ two rows
partition $\mu=(\mu_1,\mu_2)$. It follows that
\begin{eqnarray}\label{p3}
\chi(2,0;n)+\Omega(n)=2\cdot\sum_{\lm\in H(3,1;n)}\chi^\lm.
\end{eqnarray}
This implies~\eqref{rewrite3} and completes the proof of
Proposition~\ref{rewrite2}.
\end{proof}

\section{The sums $S(4,0;n)$ and $S^*(2,2;n)$}\label{s3}
\begin{defn}
\begin{enumerate}
\item
Let $n=2m$, $m\ge 2$ and let $H^*(2,2;2m)\subset H(2,2;2m)$ denote
the set of partitions $H^*(2,2;2m)=\{(k+2,k+2,2^{m-2-k})\vdash
2m\mid k=0,\ldots m-2\}$ (the partitions in the $(2,2)$ hook with
both arm and leg being rectangular), then denote
\[
S^*(2,2;2m)=\sum _{\lm\in H^*(2,2;2m)} f^\lm.
\]
\item
Let $n=2m+1$, $m\ge 2$ and let $H^*(2,2;2m+1)\subset H(2,2;2m+1)$
denote the set of partitions
$H^*(2,2;2m+1)=\{(k+3,k+2,2^{m-2-k})\vdash 2m+1\mid k=0,\ldots
m-2\}$ (the partitions in the $(2,2)$ hook with arm nearly
rectangular and leg rectangular), then denote
\[
S^*(2,2;2m+1)=\sum _{\lm\in H^*(2,2;2m+1)} f^\lm.
\]
\end{enumerate}
\end{defn}

Recall from Section~\ref{s1} that $S(4,0;2m-1)=C_m^2$ and
$S(4,0;2m)=C_m\cdot C_{m+1}$. We have the following intriguing
identities.
\begin{prop}\label{b1}
\begin{enumerate}
\item
Let $n=2m$ then \[S(4,0;2m-2)=C_{m-1}\cdot C_{m}=S^*(2,2;2m).\]
Explicitly, we have the following identity:
\begin{eqnarray}\label{b3}
C_{m-1}\cdot C_m=\frac{1}{m\cdot (m+1)}\cdot{2m-2\choose
m-1}\cdot{2m\choose
m}=~~~~~~~~~~~~~~~~~~~~~~~~~~~~~~~~~~~~~~\end{eqnarray}
\[\label{b2}~~~~~~~~~~=\sum_{k=0}^{m-2}\frac{(2m)!}{k!\cdot
(k+1)!\cdot(m-k-2)!\cdot (m-k-1)! \cdot (m-1)\cdot m^2\cdot
(m+1)}.
\]

\item
Let $n=2m+1$ then \[\frac{2m+1}{m+2}\cdot
S(4,0;2m-1)=\frac{2m+1}{m+2}\cdot C_{m}^2=S^*(2,2;2m+1).\]
 Explicitly, we have
the following identity:
\begin{eqnarray}\label{b4}
\frac{2m+1}{ m+2}\cdot C_{m}^2=
 \frac{1}{(m+1)\cdot(m+2)}\cdot{2m\choose
m}{2m+1\choose m} =~~~~~~~~~~~~~~~~~~~~~~~~~~~~~~~~~~~~~~~~~
\end{eqnarray}
\[=\sum_{k=0}^{m-2}\frac{(2m+1)!\cdot 2}{k!\cdot
(k+2)!\cdot(m-k-2)!\cdot (m-k-1)! \cdot (m-1)\cdot m\cdot
(m+1)\cdot (m+2)}. \]
\end{enumerate}
\end{prop}

\begin{proof}
Equation~\eqref{b3} is the specialization of Gauss's
$2F1(a,b;c;1)$ with $a=2-m,b=1-m, c=2$~\cite{askey},
and~\eqref{b4} is similar.
 Alternatively,
the identities~\eqref{b3} and~\eqref{b4} can be verified by the WZ
method~\cite{doron3},~\cite{doron2}.
\end{proof}

\section{Hook-sums and humps for pathes}\label{s4}

 A Dyck path of length
$2n$ is a lattice path, in $\mathbb{Z}\times \mathbb{Z}$, from
$(0,0)$ to $(2n,0)$, using up-steps $(1,1)$ and down-steps
$(1,-1)$ and never going below the $x$-axis.  A {\it hump} in a
Dyck path is an up-step followed by  a down-step.

\medskip
A  Motzkin path of length $n$ is a lattice path from $(0,0)$ to
$(n,0)$, using flat-steps $(1,0)$, up-steps $(1,1)$ and down-steps
$(1,-1)$, and never going below the $x$-axis. A hump in a Motzkin
path is an up-step followed by zero or more flat-steps followed by
a down-step.

\medskip
We count now {\it humps} for Dyck and for Motzkin paths and
observe the following intriguing phenomena: The humps-calculations
in the Dyck case correspond the $2\times n$ rectangular shape
$\lm=(n,n)$ to the $(1,1)$ hook shape $\mu=(n,1^n)$. And in the
Motzkin case we show below that it
corresponds the $(3,0)$ strip shape partitions $H(3,0;n)$ to the
$(2,1)$ hook shape partitions $H(2,1;n)$.

\subsection{The Dyck case}

The Catalan number
\[C_n=\frac{(2n)!}{n!(n+1)!}\]
is the cardinality of a variety of sets~\cite{stanley}; here we
are interested in two such sets. First, $C_n=f^{(n,n)}$, the
number of SYT of shape $(n,n)$. Second, $C_n$ is the number of
Dyck paths of length $2n$.
 Let ${\cal H} D_n$ denote
the total number of humps in all the Dyck paths of length $2n$,
then \[ {\cal H} D_n={2n-1\choose n},\]
see~\cite{dershowitz1},~\cite{dershowitz2},~\cite{deutsch}. Since
${2n-1\choose n}=f^{(n,1^n)}$, we have
\[
C_n=f^{(n,n)}\qquad\mbox{and}\qquad{\cal H} D_n=f^{(n,1^n)},
\]
which we denote by
\begin{eqnarray}\label{eq1}
{\cal H}: (n,n)\longrightarrow (n,1^n).
\end{eqnarray}

\subsection{The Motzkin case}

Like the Catalan numbers, also the Motzkin numbers $M_n$ are the
cardinality of a variety of sets; for example
$M_n=S(3,0;n)$,~\cite{regev1},~\cite{stanley},~\cite [sequence
A001006]{sloane}, which gives the Motzkin numbers a  SYT
interpretation. Also, $M_n$ is the number of Motzkin paths of
length $n$. Let ${\cal H} M_n$ denote the total number of humps in
all the Motzkin paths of length $n$, then by~\cite [sequence
A097861]{sloane}
\begin{eqnarray}\label{motzkin.path.2}
{\cal H}M_n=\frac{1}{2}\sum_{j\ge 1}{n\choose j}{n-j\choose j}.
\end{eqnarray}
We show below that this implies the intriguing identity ${\cal
H}M_n=S(2,1;n)-1,$ which gives a SYT-interpretation of the numbers
${\cal H}M_n$. Thus the humps-calculations in the Motzkin case
corresponds the $(3,0)$ strip shape partitions $H(3,0;n)$ to the
$(2,1)$ hook shape partitions $H(2,1;n)$. We denote this by
\begin{eqnarray}\label{eq2}
{\cal H}: H(3,0;n)\longrightarrow H(2,1;n).
\end{eqnarray}

\begin{thm}\label{motzkin.humps.1}
The number of humps for the Motzkin paths of length $n$ satisfies
\[
{\cal H}M_n=S(2,1;n)-1.
\]
\end{thm}

\begin{proof}
Combining Equations~\eqref{motzkin.path.3}
and~\eqref{motzkin.path.2}, the proof of
Theorem~\ref{motzkin.humps.1} will follow once  the following
binomial identity -- of interest on its own -- is proved.
\begin{lem}\label{motzkin.humps.11}
For $n\ge 2$
\begin{eqnarray}\label{motzkin.path.222}
2\sum_{j=1}^{\lfloor n/2\rfloor}{n\choose j}{n-j\choose
j}=\sum_{r=0}^{n-1}{n-r\choose{\lfloor\frac{n-r}{2}\rfloor}}
{n\choose r}+a(n)-1=~~~~~~~~~~~~~~~~~~~~~~~~~~~~~~~~~~~~~~
\end{eqnarray}
\[
~~~~~~~~~~~~=\sum_{r=0}^{n-1}{n-r\choose{\lfloor\frac{n-r}{2}\rfloor}}
{n\choose r}
+\sum_{k=1}^{\lfloor\frac{n}{2}\rfloor-1}\frac{n!}{k!\cdot
(k+1)!\cdot (n-2k-2)!\cdot (n-k-1)\cdot(n-k)}.
\]
\end{lem}

Equation~\eqref{motzkin.path.222} was verified by the WZ method.
About this method, see~\cite{doron3},~\cite{doron2}. We remark
that it would be interesting to find an elementary proof of this
identity.

\medskip
This completes the proof of Theorem~\ref{motzkin.humps.1}.
\end{proof}

A. Regev, Math. Dept. The Weizmann Institute, Rehovot 76100,
Israel.

{\it Email address:} amitai.regev at weizmann.ac.il

\end{document}